\documentclass{amsart}
\usepackage{amsmath,amssymb,amsthm}
\usepackage[all]{xy}
\newtheorem{dfn}{Definition}[section]
\newtheorem{thm}[dfn]{Theorem}
\newtheorem{lem}[dfn]{Lemma}

\newtheorem{prop}[dfn]{Proposition}

\newtheorem{conj}[dfn]{Conjecture}

\newcommand{\Hodge}{H^{2,2}(X,\mathbb{Z})}

\newcommand{\disc}{\mathrm{disc}}

\newcommand{\X}{\tilde{X}}

\newcommand{\ZZ}{\mathbb{Z}}

\newcommand{\CC}{\mathbb{C}}

\newcommand{\A}{\mathcal{A}}
\newcommand{\C}{\mathcal{C}}

\newcommand{\PP}{\mathbb{P}}
\begin{document}
\title[Certain rational cubic fourfolds]{Hilbert schemes of two points on K3 surfaces and certain rational cubic fourfolds}
\author{GENKI OUCHI}
\email{genki.ouchi@riken.jp}
\address{Interdisciplinary Theoretical and Mathematical Sciences Program (iTHEMS),
RIKEN, 2-1 Hirosawa, Wako, Saitama 351-0198, Japan
}
\maketitle
\begin{abstract}
In this paper, we check that Fano schemes of lines on certain rational cubic fourfolds are birational to Hilbert schemes of two points on K3 surfaces. 
\end{abstract}
\section{INTRODUCTION} 
\subsection{Background and Result}
The rationality problem of cubic fourfolds is one of long standing problems in algebraic geometry. The following are known examples of rational cubic fourfolds. 

\begin{itemize}
\item[(i)]special cubic fourfolds of discriminant $14$ (\cite{BRS}, \cite{Tre84}, \cite{RS})
\item[(ii)]special cubic fourfolds of discriminant $26, 38$ (\cite{RS})
\item[(iii)]some cubic fourfolds containing a plane (\cite{Has00})
\item[(iv)]some cubic fourfolds containing an elliptic ruled surface of degree $6$ (\cite{AHTVA})

\end{itemize}
Conjectually, very general cubic fourfolds are irrational. However, there are no known irrational cubic fourfolds so far. Inspired by weak factorization theorem \cite{AKMW}, \cite{W}, Kuznetsov \cite{Kuz10}, \cite{Kuz16} proposed the following conjecture.

\begin{conj}\label{Kuzconj}$($\cite{Kuz10}$)$
Let $X$ be a cubic fourfold. Consider the semiorthogonal decomposition
\[ D^b(X)= \langle \A_X, \mathcal{O}_X, \mathcal{O}_X(1), \mathcal{O}_X(2) \rangle . \]
Then $X$ is rational if and only if $\A_X$ is equivalent to a derived category of a K3 surface.
\end{conj} 
Conjecture \ref{Kuzconj} implies irrationality of very general cubic fourfolds since $\A_X$ is not equivalent to derived categories of K3 surfaces for a very general cubic fourfold $X$. Kuznetsov \cite{Kuz10} proved Conjecture \ref{Kuzconj} for Pfaffian cubic fourfolds, which are members in (i), and cubic fourfolds in (iii). Hassett studied Noether-Lefschetz divisors $\C_d$ of the moduli space $\C$ of cubic fourfolds. 
 A cubic fourfold $X$ belongs to $\mathcal{C}_d$ if and only if there is a primitive sublattice $K \subset H^{2,2}(X,\mathbb{Z})$ generated by $H^2, T \in H^{2,2}(X,\mathbb{Z})$ such that $d=\mathrm{disc}K$. Here, $H$ is the hyperplane class of $X$ and the discriminant $\mathrm{disc}K$ of $K$ is the determinant of an intersection matrix of $K$. Elements in $\C_d$ are called special cubic fourfolds of discriminant $d$. Hassett introduced the following two conditions on an integer $d$.
 
 \begin{itemize}
\item[($*$)] $d>6$ and $d \equiv 0$ or $2$ (mod $6$)
\item[($**$)]  $d$ is not divisible by $4$, $9$, or any odd prime $p \equiv 2$ (mod $3$)
\end{itemize}
The condition ($*$) is shown to be equivalent to the non-emptyness of $\C_d$ in \cite{Has00}. 
For the condition ($**$), the following theorem holds.

 \begin{prop}\label{HasK3}$($\cite{Has00}$)$\
 Let $d$ be an integer satisfying \rm{(}$*$). \it{}The integer $d$ satisfies \rm{(}$**$) \it{} if and only if for any cubic fourfold $X \in \C_d$, there is a polarized K3 surface $S$ of degree $d$ and a rank two discriminant $d$ primitive sublattice $K \subset H^{2,2}(X,\mathbb{Z})$ containing $H^2$ with a Hodge isometry $K^{\bot} \simeq H^2_{\mathrm{prim}}(S,\mathbb{Z})(-1)$. Here, the orthogonal complement $K^{\bot}$ is taken in $H^4(X,\mathbb{Z})$.
 \end{prop}
Addington and Thomas \cite{AT} compared Conjecture \ref{Kuzconj} with Proposition \ref{HasK3}. They proved that $d$ satisfies  ($*$) and ($**$) if and only if $\A_X$ is equivalent to a derived category of a K3 surface for a general $X \in \C_d$. Hassett \cite{Has00} and Addington \cite{Ad} studied the following condition.
\begin{itemize}
\item[($***$)] The equation $a^2d=2n^2+2n+2$ has an integral solution $(a,n)$. 
\end{itemize}
The condition ($***$) is related to Fano schemes of lines.  Galkin and Shinder \cite{GS} computed classes of rational cubic fourfolds in the Grothendieck ring of varieties based on the weak factorization theorem. They proved the following proposition.

\begin{prop}\label{GSthm}$($\cite{GS}$)$
Assume that the cancellation conjecture \rm{}(\cite{GS}, \rm{}Conjecture 2.7)\it{} or the weaker condition \rm{}(\cite{GS}, Remark 7.2) \it{}holds.
If a cubic fourfold $X$ is rational, then the Fano scheme of line on $X$ is birational to the Hilbert scheme of points on a K3 surface.
\end{prop}
Borisov \cite{B} found a counterexample of the cancelation conjecture. Moreover, it is known that the weaker assumption is also false \cite{IMOU}. 
Addington \cite{Ad} proved the following proposition.
 \begin{prop}\label{Adthm}$($\cite{Ad}$)$
 Let $d$ be an integer satisfying \rm{(}$*$). \it{}The integer $d$ satisfies \rm{(}$**$) \it{} if and only if the Fano scheme $F(X)$ of lines on $X$ is birational to a moduli space of stable sheaves on a K3 surface. The integer $d$ satisfies \rm{(}$***$) \it{} if and only if the Fano scheme $F(X)$ of lines on $X$ is birational to the Hilbert scheme of two points on a K3 surface.
 \end{prop}
 The condition ($***$) is stronger than ($**$). For example, the integer $74$ satisfies ($**$), but it does not satisfy ($***$).  
By Proposition \ref{Adthm}, Fano schemes of lines on rational cubic fourfolds in (i), (ii) are birational to the Hilbert schemes of points on K3 surfaces.

The main result of this paper is the following.

\begin{thm}\label{main}
Let $X$ be a rational cubic fourfold in \rm{(ii),(iii)}\it{.} 
Then the Fano scheme of lines on $X$ is birational to the Hilbert scheme of points on a K3 surface.
\end{thm}

The proof of Theorem \ref{main} will be done by computation for cohomology lattice.

\subsection{Plan of the paper}
In section 2, we study rational cubic fourfolds in (iii) and (iv) and prepare lemmas. 
In section 3, we prove Theorem \ref{main} . 

\subsection{Notation and Convention}
We work over the complex number field $\CC$. Cubic fourfolds means smooth hypersurfaces of degree three in $\PP^5$. 

\subsection*{Acknowledgements}
I would like to thanks to Nick Addington for giving the useful comment \cite{Add}. I would like to thanks to my former advisor Yukinobu Toda. The part of this work was supported by the program for Leading Graduate Schools, MEXT, Japan and Grant-in-Aid for JSPS Research Fellow 15J08505.  This work is supported by Interdisciplinary Theoretical and Mathematical Sciences Program (iTHEMS) in RIKEN. 

\section{Preliminaries} 
\subsection{Cubic fourfolds containing a plane}
Let $X$ be a cubic fourfold containing a plane $P$. Take the blow up $p: \X \to X$ of the plane $P$ in $X$. Then $\X$ has a structure of a quadric fibration
$\pi: \X \to \PP^2$ (See \cite{Has99}). Let $F$ be the cohomology class of a fiber of $\pi$ and  $Q:=p_*F$ the cohomology class of a quadric surface in $X$.
Let $K_8:=\langle H^2, Q \rangle$ be the sublattice of $\Hodge$ generated by classes $H^2$ and $Q$. The intersection matrix of $K_8$ is 
\[ \left( \begin{array}{cc}
3 & 2 \\
2 & 4 
\end{array} \right).\] 
Hassett proved the following proposition.
\begin{prop}$($\cite{Has99}$)$\label{ratsec}
The quadric fibration $\pi: \X \to \PP^2$ has a rational section if and only if there is a cohomology class $\Sigma \in \Hodge$ such that $\Sigma \cdot Q=1$. 
\end{prop}
Essentially, there are two types of intersection matrices of  rank three sublattices generated $H^2, Q$ and $\Sigma$ in Proposition \ref{ratsec}.
\begin{lem}$($\cite{AT}$)$\label{gyoretu}
Assume that there is a cohomology class $\Sigma \in \Hodge$ such that $\Sigma \cdot Q=1$. Then we can choose $\Sigma \in \Hodge$ so that 
the intersection matrix of the sublattice $\langle H^2, Q, \Sigma \rangle$ is
\[ \left( \begin{array}{ccc}
3 & 2 & 0 \\
2 & 4 & 1 \\
0 & 1 & 2k
\end{array} \right) \]
or
\[ \left( \begin{array}{ccc}
3 & 2 & 1 \\
2 & 4 & 1 \\
1 & 1 & 2k+1
\end{array} \right) \]
for some integer $k$. The first case occurs when $\mathrm{disc}\langle H^2, Q, \Sigma \rangle = 16k-3$ and the second case occurs when  $\mathrm{disc}\langle H^2, Q, \Sigma \rangle = 16k+5$.
\end{lem}
\subsection{Cubic fourfolds containing an elliptic ruled surface}  
 Let $X$ be a cubic fourfold containing an elliptic ruled surface $T$ of degree $6$. Take the blow up $p: \X \to X$ of the elliptic ruled surface $T$ in $X$. Then $\X$ has a structure of a del Pezzo fibration $\pi: \X \to \PP^2$ (See \cite{AHTVA}). Let $E$ be the exceptional divisor of $p$. Then we have the following diagram. 
 \[ \xymatrix{E \ar@{^{(}-{>}} [r]^j \ar[d]_q& \tilde{X} \ar[d]_p \ar[rd]_\pi && \\  T \ar@{^{(}-{>}}[r] & X & \mathbb{P}^2 } \] 
 
 The cohomology group of $\X$ has the decomposition 
\[H^*(\X,\mathbb{Z})=p^*H^4(X,\mathbb{Z}) \oplus j_*q^*H^2(T,\mathbb{Z}).\] 
 Let  $F$ be the cohomology class of  a fibre of $\pi$ and  $S:=p_*F $ the cohomology class of a sextic del Pezzo surface in $X$.  Let $K_{18}:=\langle H^2, S \rangle$ be the sublattice of $\Hodge$ generated by classes $H^2$ and $S$. The intersection matrix of $K_{18}$ is 
\[ \left( \begin{array}{cc}
3 & 6 \\
6 & 18 
\end{array} \right).\] 

The following proposition is an analogue of Proposition \ref{ratsec}.
\begin{prop}$($\cite{Add}, \cite{AHTVA}$)$
The del Pezzo fibration $\pi : \X \to \PP^2$ has a rational section if and only if there is a cohomology class $\Sigma \in \Hodge$ such that $\Sigma \cdot S=1,2$.
\end{prop}
\begin{proof}
By Proposition 8 in \cite{AHTVA}, $\pi$ has a rational section if and only if there exists a cohomology class $\Sigma' \in H^{2,2}(\X,\mathbb{Z})$ such that $\mathrm{gcd}(\Sigma' \cdot F, 6)=1$. The condition $\mathrm{gcd}(\Sigma' \cdot F, 6)=1$ is nothing but $\Sigma' \cdot F \equiv 1,5$ (mod $6$). 
So it is enough to show that the following two conditions are equivalent.
\begin{itemize}
\item[(1)] There exists a cohomology class $\Sigma \in \Hodge$ such that $\Sigma \cdot S=1,2$. 
\item[(2)] There exists a cohomology class $\Sigma' \in H^{2,2}(\X,\mathbb{Z})$ such that $\Sigma' \cdot F \equiv 1,5$ (mod $6$).
\end{itemize}

The second cohomology group $H^2(T,\mathbb{Z})$ is generated by the class of a section $e$ and the class of a fibre $f$ of the elliptic ruled surface $T$.  
Then we have $e^2=f^2=0$, $ef=1$ and $K_T=-2e$. Let $H_T$ be the restriction of the hyperplane class $H$ to $T$. By the proof of Theorem 2 in \cite{AHTVA}, $H_T^2=6, H_T \cdot K_T=6$ hold. So we get $H_T=e+3f$. Let $D:=S \cap T$ be the divisor class of $T$. Like the proof of Theorem 2 in \cite{AHTVA}, we can obtain $K_T=H_T-D$, that is, $D=3e+3f$. So we have $p^*S=F+j_*q^*(3e+3f)$.

First, we show that (1) implies (2). If $\Sigma \cdot S=1$ holds, we can see that $\Sigma' \cdot F=1$, where $\Sigma'=p^*\Sigma$, by the projection formula. If $\Sigma \cdot S=2$ holds, we can see that $\Sigma' \cdot F \equiv 5$ (mod $6$), where $\Sigma'=p^*\Sigma+j_*q^*e$, by the direct computation.  Next, we show the converse. There are a cohomology class $\Sigma \in \Hodge$ and $a, b \in \mathbb{Z}$ such that $\Sigma'=p^*\Sigma+j_*q^*(ae+bf)$. Then we have $\Sigma' \cdot F=\Sigma \cdot S -3(a+b)$. Since $\Sigma' \cdot F \equiv 1, 5$ (mod $6$), we have $\Sigma \cdot S \equiv 1,2,4, 5$ (mod $6$). Replacing $\Sigma$ by $-\Sigma$, we have $\Sigma \cdot S \equiv 1,2$ (mod $6$). Due to $H^2 \cdot S=6$, adding multiples of $H^2$, we can take $\Sigma$ such that $\Sigma \cdot S=1,2$. 

\end{proof}

The following lemma is an analogue of Lemma \ref{gyoretu}.
\begin{lem}\label{gyoretu2}
Assume that there is a cohomology class $\Sigma \in \Hodge$ such that $\Sigma \cdot S=c$, where $c=0,1$. Then we can choose $\Sigma \in \Hodge$ so that 
the intersection matrix of the sublattice $\langle H^2, S, \Sigma \rangle$ is 
\[ \left( \begin{array}{ccc}
3 & 6 & 0 \\
6 & 18 & c \\
0 & c & 2k
\end{array} \right) \]
or
\[ \left( \begin{array}{ccc}
3 & 6 & 1 \\
6 & 18 & c\\
1 & c & 2k+1
\end{array} \right) \]
or
\[ \left( \begin{array}{ccc}
3 & 6 & 2 \\
6 & 18 & c \\
2 & c & 2k
\end{array} \right) \]
for some integer $k$. These three cases occur when $\mathrm{disc}\langle H^2, S, \Sigma \rangle =36k-3c^2, 36k-3c^2+12c, 36k-3c^2+24c-72$ respectively.
\end{lem}
\begin{proof}
Write $H^2 \cdot \Sigma=3a+b$, where $a,b \in \mathbb{Z}$ and  $0 \leq b \leq2$. Replacing $\Sigma$ by $\Sigma-3aH^2+aS$,
we may assume that the intersection matrix is
\[ \left( \begin{array}{ccc}
3 & 6 & b\\
6 & 18 & c \\
b & c & \Sigma^2
\end{array} \right). \]
Since $\mathcal{C}_{\disc \langle H^2, \Sigma \rangle}$ is non-empty, we have 
$\disc \langle H^2, \Sigma \rangle = 3\Sigma^2-b^2 \equiv 0, 2$ (mod $6$).
We have the following:
\begin{itemize}
\item $\Sigma^2=2k$ for some $k \in \ZZ$ if $b=0$
\item $\Sigma^2=2k+1$ for some $k \in \ZZ$ if $b=1$
\item $\Sigma^2=2k$ for some $k \in \ZZ$ if $b=2$. 
\end{itemize}
Computing $\disc \langle H^2, S, \Sigma \rangle$ for $b=0,1,2$,  we obtain the desired result.
\end{proof}

 \section{Main results}
 In this section, we give the proof of the main result. We use the notation in section 3.
 Theorem \ref{main} is deduced from Proposition \ref{Adthm} and the following theorem.  
 \begin{thm}
 Assume one of the followings:
\begin{itemize}
\item[\rm{(1)}]\it{}A cubic fourfold $X$ contains a plane $P$ with a cohomology class $\Sigma \in \Hodge$ such that $\Sigma \cdot Q=1$.
\item[\rm{(2)}]\it{}A cubic fourfold $X$ contains an elliptic ruled surface $T$ of degree $6$ with a cohomology class $\Sigma \in \Hodge$ such that $\Sigma \cdot S=1$.
\item[\rm{(3)}]\it{}A cubic fourfold $X$ contains an elliptic ruled surface $T$ of degree $6$ with a cohomology class $\Sigma \in \Hodge$ such that $\Sigma \cdot S=2$.
\end{itemize}
Then there is an integer $d$ satisfying $(***)$ such that $X \in \C_d$. In particular, $F(X)$ is birational to the Hilbert scheme of two points on a K3 surface. 
 \end{thm}
 \begin{proof}
  (1) Choose $\Sigma$ as in Lemma \ref{gyoretu}. Let $\Sigma(x,y):=xQ+y\Sigma \in \langle H^2, Q, \Sigma \rangle$ and $d(x,y)=\mathrm{disc}\langle H^2, \Sigma(x,y) \rangle$ for integers $x,y\in \mathbb{Z}$. Assume that $\mathrm{disc}\langle H^2, Q, \Sigma \rangle = 16k-3$. Then we have $d(x,y)=8x^2+6xy+6ky^2$. Consider the equation
\[ a^2(8x^2+6xy+6ky^2)=2n^2+2n+2 .\]
For example, $(a,x,y,n)=(1,1-3k,1,2(1-3k))$ gives an integral solution of this equation. Assume that $\mathrm{disc}\langle H^2, Q, \Sigma \rangle = 16k+5$. Then we have $d(x,y)=8x^2+2xy+(6k+2)y^2$. Similarly, $(a,y,z,n)=(1,3k,1,6k)$ gives an integral solution of the equation
\[ a^2(8x^2+2xy+(6k+2)y^2)=2n^2+2n+2. \]
So we can take an integer $d$ satisfying ($***$) such that $X \in \C_d$.
The proofs on the cases of (2) is similar. 

(2) Using Lemma \ref{gyoretu2} as in the proof of (1), it is enough to prove the existence of integral solutions of the following equations:
\begin{itemize}
\item[(a)]$a^2(18x^2+6xy+6ky^2)=2n^2+2n+2$
\item[(b)]$a^2(18x^2-6xy+(6k+2)y^2)=2n^2+2n+2$
\item[(c)]$a^2(18x^2-18xy+(5k-4)y^2)=2n^2+2n+2$.
\end{itemize}
For examples, there are following integral solutions:
\begin{itemize}
\item[(a)]$(a,x,y,n)=(1,4k-1,2,12k-2)$
\item[(b)]$(a,x,y,n)=(1,4k+1,2,12k+2)$
\item[(c)]$(a,x,y,n)=(1,4k-5,2,12k-18)$.
\end{itemize}

(3) Using Lemma \ref{gyoretu2} as in the proof of (1) and (2), it is enough to prove the existence of integral solutions of the following equations:
\begin{itemize}
\item[(d)]$a^2(18x^2+12xy+6ky^2)=2n^2+2n+2$
\item[(e)]$a^2(18x^2+(6k+2)y^2)=2n^2+2n+2$
\item[(f)]$a^2(18x^2-12xy+(6k-4)y^2)=2n^2+2n+2$.
\end{itemize}
For examples, there are following integral solutions:
\begin{itemize}
\item[(d)]$(a,x,y,n)=(1,k-1,1,3k-2)$
\item[(e)]$(a,x,y,n)=(1,k,1,3k)$
\item[(f)]$(a,x,y,n)=(1,k-1,1,3k-4)$.
\end{itemize}

 \end{proof}

\end{document}